\documentclass[a4paper,12pt]{amsart}
\usepackage{hyperref,color}
\hypersetup{colorlinks=true,linkcolor=blue,citecolor=green}
\usepackage{amsmath,amstext,amssymb,amsthm}
\usepackage{mathtools}
\mathtoolsset{showonlyrefs,showmanualtags}

\newtheorem{theorem}{Theorem}
\newtheorem*{theorem*}{Theorem}
\newtheorem*{theorema}{Theorem A}

\newtheorem*{definition*}{Definition}
\newtheorem{proposition}[theorem]{Proposition}
\newtheorem*{proposition*}{Proposition}
\newtheorem{lemma}[theorem]{Lemma}
\newtheorem*{lemma*}{Lemma}

\newtheorem*{remark*}{Remark}

\newtheorem*{conjecture*}{Conjecture}
\numberwithin{theorem}{section}

\usepackage{enumitem}
\usepackage[margin=30truemm]{geometry}
\usepackage{bm}
\usepackage{mathrsfs}
\usepackage{stmaryrd}
\usepackage{etoolbox}

\usepackage[final]{showlabels}

\numberwithin{theorem}{section}
\numberwithin{equation}{section}

\renewcommand{\leq}{\leqslant}
\renewcommand{\geq}{\geqslant}

\DeclareMathOperator{\Int}{Int}

\providecommand\given{}

  \DeclarePairedDelimiterX\Set[1]\{\}{%
   \renewcommand\given{\SetSymbol[\delimsize]}
   #1
}

\DeclarePairedDelimiterX\norm[1]\lVert\rVert{
  \ifblank{#1}{\:\cdot\:}{#1}
}

\DeclarePairedDelimiterX\abs[1]\lvert\rvert{
  \ifblank{#1}{\:\cdot\:}{#1}
}

\DeclarePairedDelimiterX\innerp[2]{\langle}{\rangle}{#1,#2}

\DeclarePairedDelimiterXPP\Expectation[1]{\mathbb{E}}(){}{
   \renewcommand\given{\nonscript\,\delimsize\vert\nonscript\,\mathopen{}}
   #1
}

\DeclarePairedDelimiterXPP\muExpectation[2]{\mathbb{E}_{#2}}(){}{
   \renewcommand\given{\nonscript\,\delimsize\vert\nonscript\,\mathopen{}}
   #1
}

\DeclarePairedDelimiterXPP\Information[2]{I_{#2}}(){}{
   \renewcommand\given{\nonscript\,\delimsize\vert\nonscript\,\mathopen{}}
   #1
}

\DeclarePairedDelimiterXPP\Entropy[2]{H_{#2}}(){}{
   \renewcommand\given{\nonscript\,\delimsize\vert\nonscript\,\mathopen{}}
   #1
}

\AtBeginDocument{%
   \def\MR#1{}
}

\title{Density of periodic measures for piecewise monotonic maps}
\author{Ryuji Tazume}
\address{Joint Graduate School of Mathematics for Innovation, Kyushu University, 744 Motooka, Nishi-ku, Fukuoka, 819-0395, Japan}
\email{tazume.ryuji.962@s.kyushu-u.ac.jp}
\thanks{
  The author would like to express their gratitude to Masato Tsujii for fruitful discussions.
  The author was supported by WISE program (MEXT) at Kyushu University.
}
\date{\today}
\begin{document}
\begin{abstract}
	We show that any ergodic measure for a piecewise monotonic map with positive metric entropy is approximated by periodic measures in the weak-* sense.
	This partially answers Hofbauer-Raith's conjecture.
\end{abstract}
\maketitle
\section{Introduction}
This paper investigates the density of periodic measures in the space of invariant measures for a piecewise monotonic map defined on the unit interval $[0,1]$.
This property plays important roles in the study of piecewise monotonic maps from a thermodynamic point of view and has been used in various situations.
F. Hofbauer used the density of periodic measures to modify the definition of pressure (see, for instance, \cite{MR1366311}).
In \cite{Nakano-Yamamoto}, Y. Nakano and K. Yamamoto showed that the irregular set of a transitive piecewise monotonic map is either empty or has full topological entropy if the set of periodic measures is dense in the ergodic invariant measures.
Furthermore, in \cite{Chung-Yamamoto}, Y. M. Chung and K. Yamamoto proved that a piecewise monotonic map satisfying both the irreducibility of the Markov diagram and the density of periodic measures also satisfies the large deviation principle.
The detailed explanation of how the density condition of periodic measures is used is provided in their paper (see, for example, \cite[Remark 1]{Nakano-Yamamoto}, \cite[Section 3]{Chung-Yamamoto}).

The density of periodic measures has been studied by many researchers.
The condition holds for a broad class of transitive piecewise monotonic maps $T\colon[0,1]\to[0,1]$ with positive topological entropy.
For example, it holds for the following three cases:
\begin{itemize}
  \item The map $T$ is continuous.
  It follows from \cite[Corollary 10.5]{Blokh} that the space of periodic measures $M_T^\mathrm{per}([0,1])$ is dense in the space of $T$-ergodic measures $M_T^\mathrm{erg}([0,1])$ in this case.
  \item The map $T$ is a \emph{monotonic mod one transformation} (that is, there exists a strictly increasing and continuous function $f\colon[0,1]\to\mathbb{R}$ such that $T(x)=f(x)\pmod{1}$).
  The density of periodic measures was proven in \cite[Theorem 2]{Hofbauer1}.
  \item The map $T$ has two intervals of monotonicity.
  That is, there is a point $0<a<1$ such that both $T|_{(0,a)}$ and $T|_{(a,1)}$ are strictly monotonic and continuous.
  The density of periodic measures in this case was shown in \cite[Theorem 2]{Hofbauer-Raith}.
\end{itemize}
As of the time of writing, there is no known transitive piecewise monotonic map with positive topological entropy that lacks the density of periodic measures.
It is conjectured that the set of periodic measures is dense in the set of ergodic measures for all transitive piecewise monotonic maps with positive topological entropy.
This problem was originally posed by Hofbauer and Raith in their paper \cite{Hofbauer-Raith}.
Notably, in our paper, we provide a partial result of this problem.
\begin{theorema}
  Let $T\colon[0,1]\to[0,1]$ be a transitive piecewise monotonic map.
  If $\mu\in M_T^\mathrm{erg}([0,1])$ has positive metric entropy $h_\mu(T)>0$, then $\mu$ belongs to the weak-* closure of $M_T^\mathrm{per}([0,1])$.
\end{theorema}
\section{Preliminaries}
\subsection{Properties of piecewise monotonic maps}
We denote by $\mathbb{N}$ the set of all non-negative integers and put $\mathbb{N}^+=\mathbb{N}\setminus\Set{0}$.
Let $T\colon X\to X$ be a Borel measurable map on a compact metric space $X$.
We say that $T$ is \emph{transitive} if there exists a point $x\in X$ whose forward orbit $\Set{T^nx\given n\in\mathbb{N}}$ is dense in $X$.
We denote by $M(X)$ the set of all Borel probability measures on $X$ endowed with weak-* topology, by $M_T(X)\subset M(X)$ the set of all $T$-invariant measures, and by $M_T^\mathrm{erg}(X)\subset M_T(X)$ the set of $T$-ergodic measures.
A probability measure $\mu$ is called a \emph{periodic measure} if there is a periodic point $p\in X$ of period $n$ such that $\mu=(\sum_{j=0}^{n-1}\delta_{T^jp})/n$, where $\delta_x$ is the Dirac measure at $x\in X$, and let $M_T^\mathrm{per}(X)$ be the set of periodic measures on $X$.
Note that $M_T^\mathrm{per}(X)\subset M_T^\mathrm{erg}(X)$.
Finally, a Borel measurable map $T\colon[0,1]\to[0,1]$ on the unit interval $[0,1]$ is said to be a \emph{piecewise monotonic map} if there exist integer $N>1$ and $0=c_0<c_1<\cdots<c_N=1$, which we call \emph{critical points}, such that $T|_{(c_{i-1},c_i)}$ is strictly monotonic and continuous for each $1\leq i\leq N$.
Let $T\colon[0,1]\to[0,1]$ be a piecewise monotonic map and $C=\Set{c_0,\ldots,c_N}$ be the set of citical points of $T$.
The partition by critical points is denoted by $\xi=\Set*{(c_{i-1},c_i)\given 1\leq i\leq N}$.
We denote by $R$ the set of all points for which the orbit does not contain any critical points, defined as $R=\bigcap_{n=0}^\infty T^{-n}\bigcup_{Z\in\xi}Z$.
For each $k\in\mathbb{N}^+$, we define
\begin{align}
  \xi_k=\bigvee_{i=0}^{k-1}T^{-i}\xi=\Set*{\bigcap_{i=0}^{k-1}T^{-i}Z_i\neq\varnothing\given Z_i\in\xi,\,0\leq i\leq k-1}.
\end{align}
For each $x\in R$, let $\xi_k(x)$ be the unique open interval of $\xi_k$ such that it contains $x$.
If $T$ is transitive, then $\xi_k(x)$ shrinks to the one point set $\Set{x}$ as $k\to\infty$ for all $x\in R$ by the following proposition.
\begin{proposition}[{\cite[Proposition 6.1]{Yamamoto}}]\label{prop:shrinking lemma}
  Let $T\colon[0,1]\to[0,1]$ be a transitive piecewise monotonic map.
  Then we have $\bigcap_{k=1}^\infty\xi_k(x)=\Set{x}$ for all $x\in R$.
\end{proposition}
Let $\mu\in M_T^\mathrm{erg}([0,1])$.
If there exists $x\in[0,1]$ such that $\mu(\Set{x})>0$, then $\mu\in M_T^\mathrm{per}([0,1])$ by the ergodicity of $\mu$.
Therefore, for $\mu\in M_T^\mathrm{erg}([0,1])\setminus M_T^\mathrm{per}([0,1])$, $\mu([0,1]\setminus R)=0$ and $\mathscr{B}([0,1])=\bigvee_{k=0}^\infty\xi_k\pmod{\mu}$ by Proposition \ref{prop:shrinking lemma}, where $\mathscr{B}([0,1])$ is the Borel $\sigma$-algebra of $[0,1]$ and $\bigvee_{k=0}^\infty\xi_k$ is the $\sigma$-algebra generated by $\bigcup_{k=0}^\infty\xi_k$, i.e., $(\xi_k)_{k=0}^\infty$ is a generator of $\mathscr{B}([0,1])$ with respect to $T$.
This fact is used later in the calculation of metric entropy.
\subsection{Metric entropy}
Let $(X,\mathscr{B},\mu)$ be a measure space, $\alpha=\Set{A_j}_{j\in J}$ a measurable partition, where $J$ is a finite or countable set.
The \emph{information function} corresponding to $\alpha$ is defined by
\begin{align}
  \Information*{\alpha}{\mu}(x)=-\sum_{j\in J}1_{A_j}(x)\log\mu(A_j),
\end{align}
and the \emph{entropy} of $\alpha$ is defined by
\begin{align}
  \Entropy{\alpha}{\mu}=\int\Information{\alpha}{\mu}(x)\,d\mu(x)=-\sum_{j\in J}\mu(A_j)\log\mu(A_j)\in[0,\infty].
\end{align}
More generally, for a $\sigma$-algebra $\mathscr{F}\subset\mathscr{B}$, the \emph{conditional information function} and the \emph{conditional entropy of $\alpha$ with respect to $\mathscr{F}$} are defined respectively by
\begin{align}
  \Information{\alpha\given\mathscr{F}}{\mu}(x)=-\sum_{j\in J}1_{A_j}(x)\log\muExpectation{1_{A_j}\given\mathscr{F}}{\mu}(x),
\end{align}
where $\muExpectation{1_{A_j}\given\mathscr{F}}{\mu}$ is the conditional expectation of $1_{A_j}$ with respect to $\mathscr{F}$, and by
\begin{align}
  \Entropy{\alpha\given\mathscr{F}}{\mu}=\int\Information{\alpha\given\mathscr{F}}{\mu}(x)\,d\mu(x).
\end{align}

For a partition $\beta$, let $\sigma(\beta)$ be the $\sigma$-algebra generated by $\beta$.
We say that $\beta$ is finer than $\alpha$ and denote $\beta\succ\alpha\pmod{\mu}$, if $\alpha\subset\sigma(\beta)\pmod{\mu}$, i.e., every element of $\alpha$ is equal to an element of $\sigma(\beta)$ up to a set of measure zero.
We write $\alpha=\beta\pmod{\mu}$ if $\beta\succ\alpha\pmod{\mu}$ and $\alpha\succ\beta\pmod{\mu}$ hold.
For partitions $\alpha$ and $\beta$, we denote their \emph{join} by $\alpha\vee\beta=\Set*{A\cap B\given A\in\alpha,B\in\beta}$.

We shall adopt the following notation for a finite or countable partition $\alpha$, invertible measurable map $T\colon X\to X$ and $-\infty<m<n<\infty$:
\begin{align}
  \alpha_m^n=T^{-m}\alpha\vee T^{-(m+1)}\alpha\vee\cdots\vee T^{-n}\alpha.
\end{align}
For simplicity, we sometimes use the following notation:
\begin{align}
  \alpha_{-\infty}^{-1}=\bigvee_{n=1}^\infty T^n\alpha,\quad\alpha_{-\infty}^\infty=\bigvee_{n=-\infty}^\infty T^n\alpha,
\end{align}
where $\bigvee_{n=1}^\infty T^n\alpha$ is the $\sigma$-algebra generated by the union of the partitions $\bigvee_{j=1}^nT^j\alpha$ for all $n\geq 1$.
The \emph{metric entropy} of a measure preserving system $(X,\mathscr{B},\mu,T)$ is defined to be
\begin{align}
  h_\mu(T)=\sup\Set*{h_\mu(T,\alpha)\given\text{$\alpha$ is a finite measurable partition}},
\end{align}
where $h_\mu(T,\alpha)=\displaystyle\lim_{n\to\infty}\frac{1}{n}\Entropy{\alpha_0^{n-1}}{\mu}$.
See \cite{MR1958753} for detailed properties of metric entropy.
\section{The proof of Theorem A}
In this section we present a proof of Theorem A.
Let $T\colon[0,1]\to[0,1]$ be a transitive piecewise monotonic map and $\mu\in M_T^\mathrm{erg}([0,1])\setminus M_T^\mathrm{per}([0,1])$.
We first show the following Lemma.
\begin{lemma}\label{lem:key lemma}
  Assume that there exists $c>0$ such that
  \begin{align}
    \mu\left(\Set*{x\in R\given\limsup_{n\to\infty}\mu(T^n\xi_n(x))\geq c}\right)>0.
  \end{align}
  Then $\mu$ belongs to the weak-* closure of $M_T^\mathrm{per}([0,1])$.
\end{lemma}
\begin{proof}
  Put
  \begin{align}
    E=\Set*{x\in R\given\limsup_{n\to\infty}\mu(T^n\xi_n(x))\geq c}.
  \end{align}
  Since $\mu$ is $T$-ergodic and since $\mu(E)>0$ from the assumption, there exists $x\in E$ such that
  \begin{align}
    \lim_{n\to\infty}\frac{1}{n}\sum_{s=0}^{n-1}\delta_{T^sx}=\mu
    \label{eq:x generates mu}
  \end{align}
  in the weak-* sense by Birkoff's pointwise ergodic theorem.
  In the following discussion, we fix $x\in E$ satisfying \eqref{eq:x generates mu}.

  In what follows, we show that there exist $n_0\in\mathbb{N}$ and $\eta>0$ such that
  \begin{align}
    B(T^{n_0}x,\eta)\subset T^n\xi_n(x)\text{ for infinitely many }n\in\mathbb{N},
    \label{eq:take a subsequence}
  \end{align}
  where $B(T^{n_0}x,\eta)$ denotes the open interval with the center $T^{n_0}x$ and radius $\eta$.
  First, there exists a subsequence $(n_i)_{i=1}^\infty\subset\mathbb{N}$ such that $\mu(T^{n_i}\xi_{n_i}(x))\geq c/2$ for all $i$ by $x\in E$.
  Put $I_i=T^{n_i}\xi_{n_i}(x)$.
  Since $I_i$ is an open interval for all $i$, it can be expressed as $(a_i,b_i)$ using the real numbers $a_i,b_i\in[0,1]$ with $a_i<b_i$.
  By taking a subsequence of $(n_i)_{i=1}^\infty$, we may assume that $a_i,b_i$ converge to $a,b\in[0,1]$, respectively.
  Since $\mu(I_i)\geq c/2>0$ for all $i$ and $\mu$ doesn't have an atom, we have $a\neq b$.
  By taking a subsequence of $(n_i)_{i=1}^\infty$ so that $a_i,b_i$ are sufficiently close to $a,b$, respectively, we can obtain the inequality
  \begin{align}
    \mu\left(\Int\bigcap_{i=1}^\infty I_i\right)>0.
    \label{eq:nested open intervals}
  \end{align}
  From \eqref{eq:x generates mu} and \eqref{eq:nested open intervals}, we can get $n_0$ and $\eta$ satisfying \eqref{eq:take a subsequence}.

  Next we prove that $\mu$ belongs to the weak-* closure of $M_T^\mathrm{per}([0,1])$.
  Let $U$ be an arbitrary open neighbourhood of $\mu$.
  By the definition of the weak-* topology, there exist $\varepsilon>0$, $K\in\mathbb{N}$ and continuous real functions $f_1,\ldots,f_K\in C([0,1])$ such that
  \begin{align}
    \Set*{\tilde{\mu}\given\abs*{\int_{[0,1]}f_j\,d\tilde{\mu}-\int_{[0,1]}f_j\,d\mu}<\varepsilon\text{ for all }j=1,\ldots,K}\subset U.
  \end{align}
  It sufficies to show that there is a periodic point $p\in[0,1]$ such that $\mu_p\in U$.
  Put $\gamma=\max_{j=1,\ldots,K}\norm{f_j}_\infty$.
  From Proposition \ref{prop:shrinking lemma} and the uniform continuity of $f_j$, we can take an $m\in\mathbb{N}$ with
  \begin{align}
    \max_{j=1,\ldots,K}\sup_{Z\in\xi_m}\sup_{y_1,y_2\in Z}\abs{f_j(y_1)-f_j(y_2)}<\frac{\varepsilon}{5}.
    \label{eq:control the second and fourth terms}
  \end{align}
  Choose a $\delta>0$ such that
  \begin{align}
    \gamma\delta\cdot\#\xi_m<\frac{\varepsilon}{5},
    \label{eq:control the third term 1}
  \end{align}
  where $\#\xi_m$ is the number of elements of $\xi_m$.
  Let $n_0\in\mathbb{N}$, $\eta>0$ and $(n_i)_{i=1}^\infty\subset\mathbb{N}$ be such that \eqref{eq:take a subsequence} holds.
  By \eqref{eq:x generates mu}, there is an $N\in\mathbb{N}$ such that
  \begin{align}
    \abs*{\frac{1}{n}\sum_{s=0}^{n-1}f_j(T^{n_0+s}x)-\int_{[0,1]}f_j\,d\mu}<\frac{\varepsilon}{5}
    \label{eq:control the first term}
  \end{align}
  for all $j=1,\ldots,K$ and $n\geq N$.
  By Proposition \ref{prop:shrinking lemma}, there is an $N_0\in\mathbb{N}$ such that $\xi_k(T^{n_0}x)\subset B(T^{n_0}x,\eta/2)$ for all $k\geq N_0$.
  Take an $i$ such that $n_i\geq n_0+\max\Set{N,N_0,\lceil m/\delta\rceil+1}$ and put $l=n_i-n_0$.
  Then we have
  \begin{align}
    \xi_l(T^{n_0}x)\subset B(T^{n_0}x,\eta/2)\subset B(T^{n_0}x,\eta)\subset T^{n_i}\xi_{n_i}(x)\subset T^l\xi_l(T^{n_0}x).
  \end{align}
  Hence, the closure of $\xi_l(T^{n_0}x)$ is contained in $T^l\xi_l(T^{n_0}x)$.
  Since $T^l\colon\xi_l(T^{n_0}x)\to T^l\xi_l(T^{n_0}x)$ is continuous, there exists a periodic point $p\in\xi_l(T^{n_0}x)$ of period $l$ by the intermediate value theorem.
  For each $Z\in\xi_m$, we have
  \begin{align}
    \abs*{\sum_{s=0}^{l-1}\delta_{T^{n_0+s}x}(Z)-\sum_{s=0}^{l-1}\delta_{T^sp}(Z)}\leq m
    \label{eq:difference of gen. x and periodic one}
  \end{align}
  since $T^{n_0+s}x$ and $T^sp$ are in tha same element of $\xi$ for all $0\leq s\leq l-1$ by $p\in\xi_l(T^{n_0}x)$.
  Dividing both sides of \eqref{eq:difference of gen. x and periodic one} by $l$, we obtain
  \begin{align}
    \max_{Z\in\xi_m}\abs*{\frac{1}{l}\sum_{s=0}^{l-1}\delta_{T^{n_0+s}x}(Z)-\mu_p(Z)}<\delta.
    \label{eq:control the third term 2}
  \end{align}
  For every $Z\in\xi_m$, choose an $\zeta_Z\in Z$.
  Fix a $j=1,\ldots,K$.
  Then
  \begin{align}
    &\abs*{\int_{[0,1]}f_j\,d\mu-\int_{[0,1]}f_j\,d\mu_p}\\
    \leq&\abs*{\int_{[0,1]}f_j\,d\mu-\frac{1}{l}\sum_{s=0}^{l-1}f_j(T^{n_0+s}x)}\\
    &\phantom{\int_{[0,1]}f_j\,d\mu}+\frac{1}{l}\sum_{s=0}^{l-1}\sum_{Z\in\xi_m}\abs*{(f_j1_Z)(T^{n_0+s}x)-f_j(\zeta_Z)\delta_{T^{n_0+s}x}(Z)}\\
    &\phantom{\int_{[0,1]}f_j\,d\mu+\frac{1}{l}\sum_{s=0}^{l-1}}+\sum_{Z\in\xi_m}\abs{f_j(\zeta_Z)}\abs*{\frac{1}{l}\sum_{s=0}^{l-1}\delta_{T^{n_0+s}x}(Z)-\mu_p(Z)}\\
    &\phantom{\int_{[0,1]}f_j\,d\mu+\frac{1}{l}\sum_{s=0}^{l-1}+\sum_{Z\in\xi_m}}+\sum_{Z\in\xi_m}\int_Z\abs{f_j(\zeta_Z)-f_j}\,d\mu_p.
  \end{align}
  By \eqref{eq:control the first term}, the first sum on the right hand side is smaller than $\varepsilon/5$ and by \eqref{eq:control the second and fourth terms} the fourth sum is smaller than $\varepsilon/5$ as well.
  Again using \eqref{eq:control the second and fourth terms}, we get
  \begin{align}
    \sum_{Z\in\xi_m}\abs*{(f_j1_Z)(T^{n_0+s}x)-f_j(\zeta_Z)\delta_{T^{n_0+s}x}(Z)}<\frac{\varepsilon}{5},
  \end{align}
  and therefore also the second sum is smaller than $\varepsilon/5$.
  We deduce by \eqref{eq:control the third term 1} and \eqref{eq:control the third term 2} that
  \begin{align}
    \sum_{Z\in\xi_m}\abs{f_j(\zeta_Z)}\abs*{\frac{1}{l}\sum_{s=0}^{l-1}\delta_{T^{n_0+s}x}(Z)-\mu_p(Z)}<\frac{\varepsilon}{5}.
  \end{align}
  This completes the proof.
\end{proof}
\begin{proof}[Proof of Theorem A]
  Let $(X,\mathscr{B},\nu,\sigma)$ be the invertible extension of the original piecewise monotonic map $([0,1],\mathscr{B}([0,1]),\mu,T)$ (See \cite{MR0143873, MR1447586}).
  We note that $h_\nu(\sigma)=h_\mu(T)>0$.
  We can check that $\eta_{-\infty}^\infty=\mathscr{B}\pmod{\nu}$ by Proposition \ref{prop:shrinking lemma}, where $\eta=\pi^{-1}\xi$ and $\pi$ is the natural projection from $X$ to $[0,1]$.
  From Kolmogorov-Sinai's generator theorem, we have $h_\nu(\sigma)=h_\nu(\sigma,\eta)$.
  Since $\sigma$ is invertible, we have $h_\nu(\sigma^{-1})=h_\nu(\sigma)$.
  Therefore,
  \begin{align}
    \int\Information*{\eta\given\eta_{-\infty}^{-1}}{\nu}\,d\nu=\Entropy{\eta\given\eta_{-\infty}^{-1}}{\nu}=h_\nu(\sigma^{-1},\eta)=h_\nu(\sigma^{-1})=h_\nu(\sigma)>0.
  \end{align}
  Put
  \begin{align}
    A=\Set*{x\in X\given\Information{\eta\given\eta_{-\infty}^{-1}}{\nu}(x)\geq\frac{h_\nu(\sigma)}{2}}.
  \end{align}
  We have $\nu(A)>0$ and there is a $t\in\mathbb{N}^+$ with $\nu(A\cap\sigma^{-t}A)>0$ since $\nu$ is $\sigma$-ergodic.
  Set $B=A\cap\sigma^{-t}A$.
  For each $x\in B$, we get
  \begin{align}
    \min\Set*{\Information*{\eta\given\eta_{-\infty}^{-1}}{\nu}(x),\Information*{\eta\given\eta_{-\infty}^{-1}}{\nu}(\sigma^tx)}\geq\frac{h_\nu(\sigma)}{2}.
  \end{align}
  Since $\Information*{\eta\given\eta_{-n}^{-1}}{\nu}\to\Information*{\eta\given\eta_{-\infty}^{-1}}{\nu}$ as $n\to\infty$ in $L^1(\nu)$ from the Martingale convergence theorem, there are $D\in\mathscr{B}$ and $N\in\mathbb{N}$ such that $\nu(D)>1-\nu(B)$ and
  \begin{align}
    \abs*{\Information*{\eta\given\eta_{-n}^{-1}}{\nu}(x)-\Information*{\eta\given\eta_{-\infty}^{-1}}{\nu}(x)}\leq\frac{h_{\nu}(\sigma)}{4}
  \end{align}
  for all $n\geq N$ and $x\in D$.
  We can take a $x\in B\cap D$ such that $\sigma^nx\in B\cap D$ holds for infinitely many $n\in\mathbb{N}$ by $\nu(B\cap D)>0$ and Poincar\'e's recurrence theorem.
  Then
  \begin{align}
    \min\Set*{\Information*{\eta\given\eta_{-n}^{-1}}{\nu}(\sigma^nx),\Information*{\eta\given\eta_{-(n+t)}^{-1}}{\nu}(\sigma^{n+t}x)}\geq\frac{h_\nu(\sigma)}{4}
  \end{align}
  holds for infinitely many $n\geq N$.
  From the definition of the information function and the choice of $x$, we have $\mu(T^n\xi_n(\pi x))<\mu(T^{n-1}\xi_{n-1}(\pi x))$ and $\mu(T^{n+t}\xi_{n+t}(\pi x))<\mu(T^{n+t-1}\xi_{n+t-1}(\pi x))$ for infinitely many $n\geq N$.
  Then we have $\partial(T^n\xi_n(\pi x))\cap C\neq\varnothing$ and $T^n\xi_n(\pi x)\cap T^{-t}C\neq\varnothing$ for infinitely many $n\geq N$.
  Since $T^n\xi_n(\pi x)$ is an open interval and $T^{-t}C$ is a finite set, we can find $c>0$ such that $\limsup_{n\to\infty}\mu(T^n\xi_n(\pi x))\geq c$.
  Therefore, by Lemma \ref{lem:key lemma}, $\mu$ belongs to the weak-* closure of $M_T^\mathrm{per}([0,1])$.
\end{proof}
\bibliographystyle{amsplain}
\bibliography{main}
\end{document}